\documentclass[times]{oupau}
\usepackage{amsmath,amssymb,amsthm} 
\usepackage{amstext,amsfonts,mathtools}
\usepackage{mathrsfs,bm,dsfont,enumerate}
\usepackage{graphicx,color,tikz}
\usepackage{caption,subcaption}
 \usepackage{hyperref}

\providecommand{\abs}[1]{\left\lvert#1\right\rvert}

\def\CF{{\widehat{\mathscr{P}}}}
\def\D{{\mathcal{D}}}
\def\S{{\mathcal{S}}}
\def\R{{\mathcal{R}}}
\def\Lop{\mathrm{L}} 
\def\Top{\mathrm{T}} 
\def\One{\mathds{1}} 
\def\C{ \mathbb{C}}

\def\N{ \mathbb{N}}
\def\R{ \mathbb{R}}

\def\drm{\mathrm{d}}

\def\Der{\mathrm{D}}
\def\FL{(-\Delta)^{\gamma/2}}

\def\loc{\mathrm{loc}}

\begin{document}

\title{Scaling Limits of Solutions of  Linear SDE  Driven by  L\'evy White Noises}

\author{Julien Fageot\affil{1} and Michael Unser\affil{1}}

\address{%
\affilnum{1}\'Ecole polytechnique f\'ed\'erale de Lausanne, Biomedical Imaging Group, Lausanne 1015, Switzerland}

\correspdetails{julien.fageot@epfl.ch}

\received{1 Month 20XX}
\revised{11 Month 20XX}
\accepted{21 Month 20XX}


\begin{abstract}
Consider a random process $s$ solution of the stochastic differential equation $\Lop s = w$ with $\Lop$ a  homogeneous operator and $w$ a multidimensional L\'evy white noise.
In this paper, we study the asymptotic effect of zooming in or zooming out of the process $s$. More precisely, we give sufficient conditions on $\Lop$ and $w$ such that $a^H s(\cdot / a)$ converges in law to a non-trivial self-similar process for some $H$, when $a \rightarrow 0$ (coarse-scale behavior) and $a \rightarrow \infty$ (fine-scale behavior).
The parameter $H$ depends on the homogeneity order of the operator $\Lop$ and the Blumenthal-Getoor and Pruitt indices associated to the L\'evy white noise $w$. Finally, we apply our general results to several notorious classes of  random processes and random fields
and illustrate our results on simulations of L\'evy processes.
\end{abstract}

\maketitle



\section{Introduction}

A random process $s$ is self-similar if there exists $H$, called the {self-similarity index} of $s$, such that the rescaled process $a^H s(\cdot / a)$ has the same probability law than $s$ for every $a>0$.
A L\'evy process is a stochastically continuous random process $X = (X(t))_{t\in \R}$ that vanishes at $0$ and with stationary and independent increments. When the marginals of $X(t)$ are symmetric-$\alpha$-stable (S$\alpha$S), the process $X$ is self-similar   \cite{Taqqu1994stable}. More precisely, for the S$\alpha$S process $X_{\alpha}$ with $0< \alpha \leq 2$, we have that
\begin{equation}
		a^{1/\alpha} X_\alpha(t / a) \overset{(d)}{=} X_\alpha(t),
\end{equation}
for every $t \in \R$ and $a>0$. The self-similarity index of $X_\alpha$ is therefore $H = 1 / \alpha$. The case $\alpha = 2$ corresponds to the Brownian motion.
However, the L\'evy process $X$ is no longer self-similar when the noise is not stable.

The study of self-similar processes (indexed by $t \in \R$)  and fields (indexed by $\bm{x} \in \R^d$ with $d \geq 2$) is a branch of probability theory  \cite{Embrechts2000introduction}. 
They have been applied in areas such as  
signal and image processing \cite{Blu2007self,Fageot2015wavelet,Pesquet2002stochastic} 
or traffic networks \cite{Leland1994self,Mikosch2002network}, among   others \cite{Mandelbrot1982fractal,Mandelbrot1997self}. 
Many prominent random processes are self-similar, including fractional Brownian motion \cite{Mandelbrot1968}, its higher-order extensions \cite{Perrin2001nth}, infinite-variance stable processes \cite{Taqqu1994stable}, 
and their fractional versions \cite{Huang2007fractional}. 
The case of random fields have also been investigated both in the Gaussian \cite{Bentkus1981selfsimilar,Bierme2017invariance,Dobrushin1979gaussian,Lodhia2016fractional,TaftifBvf} 
and the $\alpha$-stable case  \cite{Ayache2007local,Bentkus1981selfsimilar,Bierme2007operator}. 

As already seen below, the self-similarity is intimately linked with stable laws \cite{Taqqu1994stable}, since they are  the only possible probabilistic limits of the renormalized sum of independent and identically distributed random variables. This is the well-known (generalized) central-limit theorem \cite[Section XVII-5]{Feller2008introduction}, with the consequence that self-similar processes are often scaling limits of many discretization schemes and stochastic models \cite{Bierme2010,Breton2009rescaled,Dombry2009discrete,Kaj2007scaling,Sinai1976self}. 

The self-similarity imposes a strong constraint on the law of the random process. In particular, it intimately links the behaviors at coarse and fine scales.
We have mentioned how the self-similar models have been successfully used,  but, it can also appear to be too restrictive.
One advantage of the family of L\'evy processes and their generalizations is to overcome this restriction. 

In this paper, we focus on the impact of rescaling operations for a  broad class of random processes that are asymptotically or locally self-similar.
These processes are specified as the solutions of a stochastic differential equation  of the form 
\begin{equation} \label{eq:main}
\Lop s = w,
\end{equation}
where $w$ a multidimensional L\'evy white noise and $\Lop$ is a differential operator on the functions from $\R^d$ to $\R$. We assume moreover that the operator $\Lop$ is homogeneous of some order $\gamma \geq 0$, in the sense that $\Lop \{ \varphi ( \cdot / a ) \} = a^{-\gamma} ( \Lop \varphi ) (\cdot / a)$ for any function $\varphi$ and $a>0$. Typically, in dimension $d=1$, the derivative is homogeneous of order $1$.
Our aim is to study the statistical behavior of the rescaling $\bm{x} \mapsto s(\bm{x} / a)$ of a solution of \eqref{eq:main} when $a>0$ is varying. Our two main questions are:
\begin{itemize}
	\item What is the asymptotic behavior of $s( \cdot / a )$ when we zoom out the process (\emph{i.e.}, when $a \rightarrow 0$)?
	\item What is the local behavior when we zoom in (\emph{i.e.}, when $a \rightarrow \infty$)?
\end{itemize} 

Our main contribution is to identify sufficient conditions such that the rescaling $a^H s(\cdot / a)$ of a solution of \eqref{eq:main} has a non-trivial self-similar asymptotic limit as $a$ goes to $0$ or $\infty$. When this limit exists, the parameter $H$ is unique and depends essentially on the degree of homogeneity $\gamma$ of $\Lop$ and on the Blumenthal-Getoor and Pruitt indices $\beta_\infty$ and $\beta_0$ of $w$ \cite{Blumenthal1961sample,Pruitt1981growth}. The indices $\beta_0$ and $\beta_\infty$ are used in the literature to characterize the asymptotic and local  behaviors of L\'evy processes, and more generally L\'evy-type processes \cite{Bottcher2014levy}.
We summarize the main results of our paper in Theorem \ref{theorem:summary}. Precise definitions and rigorous statements are given later.
\begin{theorem} \label{theorem:summary}
	Let $\Lop$ be a $\gamma$-homogeneous operator and $w$ a L\'evy process with indices $\beta_\infty$ and $\beta_0$. Under some technical conditions, the solution $s$ to the equation $\Lop s = w$ has the following properties. 
	\begin{itemize}
		\item The process $s$ is asymptotically self-similar of order $H_\infty = \gamma + d\left( 1 / \beta_0 - 1\right)$, in the sense that $a^{H_\infty} s(\cdot / a)$ converges to a self-similar process of order $H_\infty$ when $a \rightarrow 0$.
		\item The process $s$ is locally self-similar of order $H_{\mathrm{loc}} = \gamma + d\left( 1 / \beta_\infty - 1\right)$, in the sense that $a^{H_{\mathrm{loc}}} s(\cdot / a)$ converges to a self-similar process of order $H_{\mathrm{loc}}$ when $a \rightarrow \infty$.
	\end{itemize}
\end{theorem}

The paper is organized as follows: In Section \ref{sec:GRP}, we introduce the framework of generalized random processes, while  the general class of random processes of interest is addressed  in Section \ref{sec:SDE}. We precise and demonstrate Theorem \ref{theorem:summary} in Section \ref{sec:theorems}, where we identify sufficient conditions under which the asymptotic behavior of $a^H s(\cdot / a)$ is known at coarse and fine scales. We also investigate the necessity of these conditions. Finally, we apply our results  in Section \ref{sec:examples} to specific classes of random processes, for   different types of white noises and operators.
 
\section{Generalized Random Processes} \label{sec:GRP}

The theory of generalized random processes was introduced independently in the 50's by K. It\^o  \cite{Ito1954distributions} and I. Gelfand \cite{Gelfand1955generalized}. Among the benefits of this theory to  the construction and study of random processes, we mention:
\begin{itemize}
	\item \textit{Its generality.} It allows one to define the broadest class of linear processes, including processes with no pointwise interpretation such as L\'evy white noises.
	\item \textit{The availability of an infinite-dimensional Bochner theorem.} The characteristic functional (see Definition \ref{def:CF}) characterizes the law of a generalized random process in the same way  the characteristic function does for random variables.  This allows for the construction of generalized random processes via their characteristic functional (see Theorem \ref{theo:MinlosBochner} below). 
	\item \textit{The availability of an infinite-dimensional L\'evy continuity theorem.} The convergence in law of a sequence of random vectors is equivalent to the pointwise convergence of the corresponding characteristic functions. This result is also true  for the generalized random processes defined over the extended space of tempered distribution $\S'(\R^d)$ (see Theorem \ref{theo:BoulicautLevy}). This provides a powerful tool to show the convergence in law of random processes. We shall  exploit this tool extensively in this paper. 
	\end{itemize}
	
	\subsection{Generalized Random Processes and their Characteristic Functional} \label{subsec:definition}

The space of rapidly decaying and infinitely differentiable functions is denoted by $\S(\R^d)$ and endowed with its usual Fr\'echet topology. Its continuous dual, the space of tempered distribution, is $\S'(\R^d)$. We fix a probability space $(\Omega, \mathcal{F}, \mathscr{P})$. The space of real-valued  random variables $L^0(\Omega)$ is endowed with the Fr\'echet topology associated with the convergence in probability. 

\begin{definition} \label{def:GRP}
	A \emph{generalized random process} $s$ on $\S'(\R^d)$  is mapping $\varphi \mapsto \langle s , \varphi \rangle$  from $\S(\R^d)$ to $L^0(\Omega)$ that is linear and sequentially continuous in probability.
\end{definition}

The generalized random process $s$  in Definition \ref{def:GRP} is specified as a random linear functional over the space $\S(\R^d)$. 
Alternatively, one can see $s$ as a random variable with values in $\S'(\R^d)$; that is, a measurable map from $\Omega$ to $\S'(\R^d)$, where $\S'(\R^d)$ is endowed with the Borelian $\sigma$-field associated with the strong topology. 
This equivalence is a consequence of the structure of the nuclear Fr\'echet space of $\S(\R^d)$ \cite{Fernique1967processus,Ito1984foundations}.
It means in particular that one should look at $s$ as a random tempered generalized function. 

\begin{definition} \label{def:CF}
	The \emph{characteristic functional} of a generalized random process $s$ is defined  as
$\CF_s(\varphi) = \mathbb{E} \left[ \mathrm{e}^{\mathrm{i} \langle s ,\varphi\rangle} \right]$ for every $\varphi \in \S(\R^d)$.	
\end{definition}

As announced in the Introduction of Section \ref{sec:GRP}, we present the generalizations of the Bochner and L\'evy continuity theorems for generalized random processes.

\begin{theorem}\label{theo:MinlosBochner}
	A functional $\CF : \S(\R^d) \rightarrow \C$ is the characteristic functional of a generalized random process if and only if  it is continuous and positive-definite over $\S(\R^d)$,  and normalized as $\CF(0) = 1$.
\end{theorem}

This result  is known as the Minlos-Bochner theorem. It is valid  for any functional defined over a nuclear Fr\'echet space such as $\S(\R^d)$ \cite{Minlos1959generalized}, and more generally over an inductive limit of nuclear Fr\'echet spaces such as $\D(\R^d)$, the space of compactly supported smooth functions \cite[Theorem II.3.3]{Fernique1967processus}. The nuclear structure of $\S(\R^d)$ is at the heart of the theory of generalized random processes.

\begin{definition}
	A sequence of generalized random processes $(s_n)$ converges in law to the generalized process $s$ if, for any $\varphi_1, \ldots ,\varphi_N \in \S(\R^d)$, the sequence of  random vectors  $  ( \langle s_n , \varphi_1 \rangle , \ldots , \langle s_n , \varphi_N \rangle )  $ converges in law to the random vector $ ( \langle s , \varphi_1 \rangle , \ldots , \langle s , \varphi_N \rangle ) $.
\end{definition}

\begin{theorem}\label{theo:BoulicautLevy}
	A sequence of generalized random processes $(s_n)$ converges in law to the generalized random process $s$ if and only if 
$\CF_{s_n}(\varphi) \underset{n \rightarrow \infty}{\longrightarrow} \CF_s(\varphi)$
for any $\varphi \in \S(\R^d)$. 
\end{theorem}

This result was proved by X. Fernique for the space of compactly supported smooth functions $\D(\R^d)$ \cite[Theorem III.6.5]{Fernique1967processus}. 
The case of  a nuclear Fr\'echet space (including $\S(\R^d)$) is simpler and can be deduced from the results of Fernique. It is also proved by P. Boulicaut in \cite{Boulicaut1973} together with a converse to this result: The weak convergence of probability measures on the dual of a Fr\'echet space $\mathcal{F}$ is equivalent with the pointwise convergence of the characteristic functionals on $\mathcal{F}$ if and only if $\mathcal{F}$ is nuclear \cite[Theorem 5.3]{Boulicaut1973}.
A comprehensive and self-contained exposition of the theory of generalized random processes, including proofs of Theorems \ref{theo:MinlosBochner} and \ref{theo:BoulicautLevy}, can be found in \cite{Bierme2017generalized}. 

	\subsection{L\'evy White Noises and Infinite Divisibility}

The construction of continuous-domain L\'evy white noises and related processes  is intimately linked with the infinite divisibility of the finite-dimensional marginals of those processes \cite{GelVil4,Sato1994levy}. 
A random variable (and more generally a random vector) is \emph{infinitely divisible} if it can be decomposed (in law) as the sum of $N$ i.i.d. random variables for every $N$. The logarithm of the characteristic function $\CF_X$ of an infinitely divisible random variable $X$ is called its \emph{L\'evy exponent}, denoted by $\Psi$; \emph{i.e.}, 
$\CF_{X}(\xi) = \exp (\Psi (\xi) ).$

Initially, the family of L\'evy white noises was introduced  over the space $\D'(\R^d)$ of generalized functions \cite[Chapter III]{GelVil4}. 
There is  a one-to-one correspondence between infinitely divisible random variables and L\'evy white noises in $\D'(\R^d)$.
Indeed, the characteristic functional of a L\'evy white noise is of the form $\CF_w(\varphi) = \exp( \int_{\R^d} \Psi(\varphi(\bm{x})) \mathrm{d}\bm{x} )$ where $\Psi$ is a L\'evy exponent. 
However, a L\'evy white noise on $\D'(\R^d)$ is not necessarily tempered (a counter-example is given in \cite[Section 3.1]{Fageot2014}).The characterization of tempered white noises has been obtained recently and is based on Theorem \ref{theo:CFonS}.

\begin{theorem} \label{theo:CFonS}
Let $\Psi$ be a continuous function from $\R$ to $\C$. 
The functional 
\begin{equation} \label{eq:formCF}
\CF(\varphi) = \exp \left( \int_{\R^d} \Psi(\varphi(\bm{x})) \drm \bm{x} \right)
\end{equation}
is the characteristic functional of a generalized random process in $\S'(\R^d)$ if and only if  $\Psi$ is a L\'evy exponent and the infinitely divisible random variable $X$ with L\'evy exponent $\Psi$ has a finite $\epsilon$-moment for some $\epsilon > 0$, such that
$ \mathbb{E} \left[ \abs{X}^\epsilon \right] < \infty .$
\end{theorem}

We have shown that the existence of a finite absolute moment is sufficient for $w$ being tempered  in \cite[Theorem 3]{Fageot2014}. More recently, R. Dalang and T. Humeau have proved that this condition is also necessary \cite[Theorem 3.13]{Dalang2015Levy}.
This provides a one-to-one correspondence between tempered L\'evy noise and infinitely divisible random variable having a finite absolute moment, what justifies the following definition.

\begin{definition}
	A generalized random process $w$ whose characteristic functional has the form \eqref{eq:formCF} with $\Psi$ satisfying the two conditions of Theorem \ref{theo:CFonS} is called a \emph{L\'evy white noise} in $\S'(\R^d)$.
\end{definition}

The finiteness of absolute moments is strongly related to the behavior of the L\'evy exponent at the origin, or equivalently, the asymptotic behavior of the L\'evy measure $\mu$ associated to $\Psi$ (see Section \ref{subsec:examplesnoises} for a short reminder on (symmetric) L\'evy measures). Especially, the condition $\mathbb{E} \left[ \abs{X}^p \right] < \infty$ is equivalent to $\int_{\lvert t \rvert\geq 1} \lvert t \rvert^p \mu(\mathrm{d}t)< \infty$ for every $p>0$ \cite[Section 5.25]{Sato1994levy}. The point in Theorem \ref{theo:CFonS}  is that $\epsilon$ can be arbitrarily small, hence this requirement for being tempered is extremely mild and satisfied by any L\'evy noise encountered in practice.

A L\'evy white noise is stationary in the sense that $w(\cdot - \bm{x}_0)$ and $w$ have the same law for every $\bm{x}_0 \in \R^d$. It is moreover independent at every point, meaning that $\langle w ,\varphi_1 \rangle $ and $\langle w , \varphi_2 \rangle$ are independent whenever $\varphi_1$ and $\varphi_2$ have disjoint supports.

As we shall see, one particular subclass of L\'evy white noise plays a crucial role as potential scaling limits of general L\'evy white noises: the S$\alpha$S (symmetric-$\alpha$-stable) white noises.

\begin{definition}
	Let $0< \alpha \leq 2$. A L\'evy white noise $w_\alpha$ is a \emph{S$\alpha$S white noise} if its characteristic functional has the form
	\begin{equation} \label{eq:SaSCF}
		\CF_{w_\alpha} (\varphi) = \exp ( - C \lVert \varphi \rVert_\alpha^{\alpha})
		\end{equation}
		for some $C>0$ and every $\varphi \in \S(\R^d)$, where $\lVert \varphi \rVert_\alpha = \left( \int_{\R^d} \lvert \varphi (\bm{x}) \rvert^\alpha \mathrm{d} \bm{x} \right)^{1/\alpha}$.
\end{definition}

The functional \eqref{eq:SaSCF} is a characteristic functional and corresponds to \eqref{eq:formCF} with L\'evy exponent $\Psi(\xi) = - C \abs{\xi}^\alpha$.
For every $\varphi \in \S(\R^d)$, the random variable $X = \langle w_{\alpha}, \varphi \rangle$ is   S$\alpha$S with characteristic function $\CF_X (\xi) = \exp( - C \lVert \varphi \rVert_\alpha^\alpha \abs{\xi}^\alpha)$. For $\alpha = 2$, one recognizes the Gaussian law. When $\alpha<2$, by contrast, the considered random variables have infinite variances. More information on non-Gaussian S$\alpha$S random variables and processes can be found on \cite{Taqqu1994stable}. 

	\subsection{Indices of L\'evy White Noises} \label{subsec:blumenthal}

\begin{definition}  \label{def:BG}
We consider the following quantities associated to a L\'evy exponent $\Psi$: 
		\begin{align}
			\beta_0 &= \sup \left\{ p \in [0,2], \quad   \underset{|\xi| \rightarrow 0}{\lim \sup}  \frac{\abs{\Psi(\xi)}}{|\xi |^p} < \infty \right\},  \label{eq:b0} \\
			\beta_\infty &= \inf \left\{ p \in [0,2], \quad  \underset{|\xi| \rightarrow \infty}{\lim \sup}  \frac{\abs{\Psi(\xi)}}{|\xi |^p} < \infty  \right\}. \label{eq:binfty}
		\end{align}
	We call $\beta_0$ the \emph{Pruitt index} and $\beta_\infty$ the \emph{Blumenthal-Getoor index} of $\Psi$. 
\end{definition}

The Blumenthal-Getoor index $\beta_\infty$ was initially introduced   in \cite{Blumenthal1961sample} to study the   behavior of L\'evy processes at the origin.  
It is connected to the local regularity of L\'evy processes \cite{Fageot2017besov}, and more generally of Feller processes \cite{Schilling1998growth,Schilling2000function}. 
 The index $\beta_0$ is the asymptotic counterpart of $\beta_\infty$ in the sense that it relies to the behavior of L\'evy processes at infinity. 
 It was considered by B. Pruitt \cite{Pruitt1981growth} and is highly connected to the existence of moments of the infinitely divisible random variable with L\'evy exponent $\Psi$. 
 Many global regularity properties of L\'evy \cite{Fageot2017multidimensional} and Feller processes \cite{Schilling2000function} are captured by the knowledge of $\beta_0$ and $\beta_\infty$. 
Moreover, one can often characterized $\beta_0$ with the L\'evy measure associated to $\Psi$ (see \cite[Section 3]{Deng2015shift}). 
The fact that  a L\'evy white noise on $\S'(\R^d)$ always has a finite  moment $\epsilon>0$ finite (see Theorem \ref{theo:CFonS})  imposes that $\beta_0 > 0$. Consequently,  we should always consider indices such that  $0<\beta_0 \leq 2$.

Consider a L\'evy exponent $\Psi$  with  indices $0<\beta_0, \beta_\infty \leq 2$ and fix $0<p \leq2$. The notation $f(\xi) \underset{0/\infty}{\sim} g(\xi)$ means that $g(\xi) \neq 0$ for $\xi\neq 0$ and that $f(\xi) / g(\xi)$ converges to $1$ when $\xi$ goes to $0/\infty$.  From the definition, we easily see if $\Psi (\xi) \underset{\infty}{\sim} - C \abs{\xi}^p$, then $\beta_\infty = p$. Similarly, if $\Psi (\xi) \underset{0}{\sim} - C \abs{\xi}^p$, then $\beta_0 = p$. Therefore, the indices $\beta_0$ and $\beta_\infty$ are respectively the only possible power law behavior of the L\'evy exponent in the origin or asymptotically, respectively. Finally, if the L\'evy exponent is bounded as $\abs{\Psi(\xi)} \leq C \abs{\xi}^p$, then $\beta_\infty \leq p \leq \beta_0$.

\section{Linear SDE Driven by L\'evy White Noises} \label{sec:SDE}

The main goal of this section is to introduce the class of random processes of interest for the study of the local and asymptotic self-similarity.
A linear differential operator $\Lop$ and a white noise in $\S'(\R^d)$ being given, we consider the linear stochastic differential equation
\begin{equation} \label{eq:Lsw}
	\Lop s = w.
\end{equation}
We say that a solution exists if there is a generalized random process $s$ in $\S'(\R^d)$ such that the processes $\Lop s$ and $w$ are equal in law (or equivalently, the same characteristic functional).
The general framework to solve \eqref{eq:Lsw}  is  based on the existence of inverse operators with adequate properties   \cite[Chapter 4]{Unser2014sparse}. 
In this section,  we first construct  generalized random processes that are  solutions of \eqref{eq:Lsw} (Section \ref{subsec:principle}). Then, we introduce the homogeneous operators and the class of studied processes, called $\gamma$-order linear processes in Section \ref{subsec:operators}. 

	\subsection{Construction of Linear Processes} \label{subsec:principle}

Let  $\Lop$ be  a continuous and linear operator from $\S(\R^d)$ to $\S'(\R^d)$. Then, its adjoint   is the operator $\Lop^*$ from $\S(\R^d)$ to $\S'(\R^d)$ defined as
$	\langle \Lop^* \varphi_1 , \varphi_2 \rangle = \langle \varphi_1 , \Lop \varphi_2 \rangle
$ for every $\varphi_1,\varphi_2\in \S(\R^d)$.

\begin{proposition}[Specification of a linear process] \label{prop:principle}
Consider a linear and continuous operator $\Lop$ from $\S(\R^d)$ to $\S'(\R^d)$ and a L\'evy white noise $w$ on $\S'(\R^d)$.
Assume the existence of a topological vector space $\mathcal{X}$ such that 
\begin{itemize}
	\item the adjoint $\Lop^*$ of $\Lop$ admits a left inverse operator $\Top$ that is linear and continuous from $\S(\R^d)$ to $\mathcal{X}$;
	\item the characteristic functional $\CF_w$ of $w$ can be extended as a continuous and positive-definite functional on $\mathcal{X}$.
\end{itemize}
Then, there exists a generalized random process $s$ whose characteristic functional is, for every   $\varphi \in \S(\R^d)$,
$\CF_s(\varphi) = \CF_w(\Top \varphi)$. Moreover, we have that
$\Lop s \overset{(d)}{=} w$.
\end{proposition}

By considering a more general $\mathcal{X}$, this result refines the original theorem that was first presented  in  \cite[Section 3.5.4]{Unser2014sparse} albeit with some unnecessary restrictions on $\mathcal{X}$. The principle is simply to check the conditions of the Minlos-Bochner theorem. We give the proof for the sake of completeness.

\begin{proof}
Set $\CF(\varphi) = \CF_w(\Top \varphi)$. From the assumption on $\CF_w$ and $\Top$, we easily deduce that $\CF$ is well-defined and continuous over $\S(\R^d)$ by composition. It is positive-definite by composition of  linear and a positive-definite functions. Finally, $\Top \{0\} = 0$, hence $\CF(0) =  \CF_w(0) = 1$. Therefore, $\CF$ is the characteristic functional of a generalized random process $s$ according to Theorem \ref{theo:MinlosBochner}.  
For the last point, we simply remark that, $\Top$ being a left inverse of $\Lop^*$,
$	\CF_{\Lop s} (\varphi) = \mathbb{E}\left[  \mathrm{e}^{\mathrm{i} \langle \Lop s , \varphi \rangle } \right] =  \mathbb{E}\left[  \mathrm{e}^{\mathrm{i} \langle s , \Lop^*  \varphi   \rangle } \right] = \CF_w( \Top  \Lop^* \varphi )  = \CF_w(\varphi),
$ which is equivalent to  $\Lop s \overset{(d)}{=} w$. 
\end{proof}

\begin{definition}
	A generalized random process constructed via Proposition \ref{prop:principle} is called a \emph{linear process}.
\end{definition}

In practice, for given $\Lop$ and $w$, one has to determine an adequate space $\mathcal{X}$ in order to correctly define the process $s$. The choice of $\mathcal{X}$ is generally driven by the white noise. For instance, we will consider the case $\mathcal{X} = L^\alpha(\R^d)$ when dealing with S$\alpha$S white noises. The optimal choice of $\mathcal{X}$ for a given $w$ is investigated in \cite{Fageot2017unified} based on the results of \cite{Rajput1989spectral}.
Then, the main issue becomes the existence of a left inverse with the adequate stability, mapping $\S(\R^d)$ into $\mathcal{X}$.

	\subsection{Homogeneous  Operators and $\gamma$-order Linear Processes} \label{subsec:operators}

This section is dedicated to the specification of random processes concerned by Theorem \ref{theorem:summary}. 
We start with some definitions. 
For $u \in \S'(\R^d)$, $\bm{x}_0 \in \R^d$, and $a>0$, we define $u(\cdot - \bm{x}_0)$ as the tempered distribution such that $\langle u(\cdot - \bm{x}_0) ,\varphi \rangle = \langle u ,\varphi  (\cdot + \bm{x}_0)\rangle$ for every $\varphi \in \S(\R^d)$. Similarly, $u( \cdot / a)$ is the tempered distribution such that $\langle u(\cdot /a ) ,\varphi \rangle = \langle u , a^d \varphi  (a \cdot )\rangle$.

\begin{definition}
	Consider a linear and continuous operator $\Lop$ from $\S(\R^d)$ to $\S'(\R^d)$. We say that $\Lop$ is \emph{$\gamma$-homogeneous} for some $\gamma \in \R$ if $\Lop \{ \varphi ( \cdot / a ) \} = a^{-\gamma} ( \Lop \varphi ) (\cdot / a)$ 
	for every $\varphi \in \S(\R^d)$ and $a > 0$. 
\end{definition}

For instance, the derivative is a $1$-homogeneous operator.
We shall now focus on operators $\Lop$ that are: (1) linear shift-invariant, (2) continuous from $\S(\R^d)$ to $\S'(\R^d)$, and (3)  $\gamma$-homogeneous for some $\gamma \geq 0$.
Moreover, inspired by Proposition \ref{prop:principle}, the adjoint operator $\Lop^*$ should have a left inverse with some stability property. We shall essentially consider two cases, assuming the existence of a left inverse $\Top$ as in Proposition \ref{prop:principle} for $\mathcal{X} = \mathcal{R}(\R^d)$, the space of rapidly decaying measurable functions (see below), or $\mathcal{X} = L^p(\R^d)$ for some $p$ such that $0<p \leq 2$. These spaces naturally arise as domains of continuity of the characteristic functional of L\'evy white noises.

The space  $\mathcal{R}(\R^d)$ is defined as
$\mathcal{R}(\R^d)   = \left\{ f \text{ measurable}, \ (1 + \abs{\cdot})^N f \in L^2(\R^d) \text{ for all } N \in \N \right\}$. It is endowed with a natural Fr\'echet topology, as a projective limit of the Hilbert spaces $ L^2_N(\R^d) = \{ f , \ (1 + \abs{\cdot})^N f \in L^2(\R^d) \}$, $N\in \N$ (see \cite[Chapter IV]{Meise1997introduction} for more details on Fr\'echet spaces). 

Fix a linear, shift-invariant, continuous,  and $\gamma$-homogeneous operator from $\S(\R^d)$ to $\S'(\R^d)$, with $\gamma \geq 0$. We   consider two cases.
\begin{itemize}
	\item \textbf{Condition (C1).} The adjoint $\Lop^*$ admits a $(-\gamma)$-homogeneous left inverse that continuously maps $\S(\R^d)$ to $\mathcal{R}(\R^d)$. 
	\item \textbf{Condition (C2).} The adjoint $\Lop^*$ admits a $(-\gamma)$-homogeneous left inverse that continuously maps $\S(\R^d)$ to $L^p (\R^d)$ for some $0<p \leq 2$. 
\end{itemize}
Note that   (C1) is more restrictive than (C2)  since $\mathcal{R}(\R^d) \subset L^p(\R^d)$ for any $0<p\leq 2$. \\
 
One shall construct the random processes of interests thanks to Proposition \ref{prop:principle}. We start by giving some new results on the continuity of the characteristic functionals of L\'evy white noises that allow  for new compatibility conditions between an operator $\Lop$ and a L\'evy white noise $w$ in the general case and for $p$-admissible white noise (see Definition \ref{def:padmissible} below). 
 
 \begin{definition} \label{def:padmissible}
	We say that a L\'evy exponent $\Psi$ is \emph{$p$-admissible} for $0<p\leq 2$ if $\abs{\Psi(\xi)} \leq C \abs{\xi}^p$ 
	for some $C>0$ and every $\xi \in \R$.
	By extension, a L\'evy white noise with a $p$-admissible L\'evy exponent is said to be \emph{$p$-admissible} itself.
\end{definition} 

\begin{proposition} \label{prop:Rcontinuous}
		Let $w$ be a L\'evy white noise over $\S'(\R^d)$. Then, the characteristic functional $\CF_w$ of $w$ can be extended as a continuous and positive-definite functional over $\mathcal{R}(\R^d)$.
		If moreover $w$ is $p$-admissible for some $0<p\leq 2$, then $\CF_w$ 	can be extended as a continuous and positive-definite functional over $L^{p}(\R^d)$.
\end{proposition}

\begin{proof}
	The proof for $L_p(\R^d)$ is similar to the one for $\mathcal{R}(\R^d)$, hence we omit it. We only need to prove the continuity, because the positive-definiteness follows simply by density of $\mathcal{S}(\R^d)$ in $\mathcal{R}(\R^d)$.
	The positive-definiteness of $\CF_w$ in $\S(\R^d)$ implies that 
		\begin{equation}
		\abs{\CF(\varphi_2) - \CF(\varphi_1)} \leq  2 \abs{ 1 - \CF(\varphi_2 - \varphi_1) }
		\end{equation}
	for any $\varphi_1, \varphi_2 \in \S(\R^d)$ (see for instance  \cite[Section II.5.1]{Fernique1967processus}  or  \cite[Section IV.1.2, Proposition 1.1]{ProbaBanach1987}). 
For every $z \in \C$ with $\Re\{z\} \leq 0$, 
one has $\lvert \mathrm{e}^{\mathrm{z}} - 1 \rvert \leq \lvert z \rvert$. Since $\Re\{\Psi\}\leq 0$, this implies that 
\begin{equation} \label{eq:boundCF}
	\lvert 1 - \CF_w(\varphi) \rvert \leq \lvert 1 - \mathrm{e}^{\int_{\R^d} \Psi(\varphi(\bm{x}))\mathrm{d}\bm{x}}\rvert \leq \left\lvert \int_{\R^d}  \Psi (\varphi(\bm{x}))   \mathrm{d}\bm{x} \right\rvert \leq \int_{\R^d} \lvert \Psi (\varphi(\bm{x}))  \rvert \mathrm{d}\bm{x}  . 
\end{equation}
	Moreover, according to \cite[Corollary 1]{Fageot2014}, $w$ being tempered, there exists $C>0$ and $\epsilon>0$ such that $\abs{\Psi(\xi)} \leq C ( \abs{\xi}^{\epsilon}  + \abs{\xi}^2)$. Putting the ingredients together, we then easily show that
\begin{equation} \label{eq:boundCFw1w2}
	\abs{\CF_w(\varphi_2) - \CF_w(\varphi_1)}  \leq C ( \lVert \varphi_2 - \varphi_1 \rVert_\epsilon^\epsilon + \lVert \varphi_2 - \varphi_1 \rVert_2^2). 
\end{equation}
Hence, $\CF_w$ can be extended continuously to functions $\varphi \in \mathcal{R}(\R^d) \subset L^2(\R^d) \cap L^\epsilon(\R^d)$ as expected.
\end{proof}


\paragraph{Remark.} In \cite[Definition 4.4]{Unser2014sparse}, an alternative definition of the $p$-admissibility was  introduced. Since this definition requires an additional bound on the derivative of the L\'evy exponent and  is used for the construction of random processes only for $1 \leq p \leq 2$, the second part of Proposition \ref{prop:Rcontinuous} is a generalization of \cite[Theorem 8.2]{Unser2014sparse}.
Proposition \ref{prop:Rcontinuous}   allow for new criteria  to solve SDEs driven by L\'evy white noises.

\begin{theorem} \label{coro:CFwT}
Let $w$ be a L\'evy white noise on $\S'(\R^d)$ with L\'evy exponent $\Psi$ and $\Lop$ be a   linear,  $\gamma$-homogeneous, and continuous operator from $\S(\R^d)$ to $\S'(\R^d)$ for $\gamma \geq 0$. We consider two cases.
\begin{itemize}
	\item \textbf{Condition (C1):} The adjoint $\Lop^*$ admits a $(-\gamma)$-homogeneous left inverse   $\Top$  that maps $\S(\R^d)$ to $\mathcal{R}(\R^d)$.
	\item \textbf{Condition (C2):} There exists $0<p \leq 2$ such that (i) the adjoint $\Lop^*$ admits a $(-\gamma)$-homogeneous  left inverse  $\Top$ that maps $\S(\R^d)$ to $L^p(\R^d)$ and (ii) $\Psi$ is \mbox{$p$-admissible}.	
\end{itemize}
When  (C1) or (C2) are satisfied,   there exists a generalized random process $s$ whose characteristic functional is $\CF(\varphi) = \CF_w(\Top \varphi )$. Moreover, $s$ is a solution of \eqref{eq:Lsw}. 
\end{theorem}

\begin{proof}
	The result follows from the application of Proposition \ref{prop:principle}  with $\mathcal{X} = \mathcal{R}(\R^d)$ and $\mathcal{X} = L^p(\R^d)$, respectively. The assumptions on $\CF_w$ are satisfied due  to Proposition \ref{prop:Rcontinuous}.
\end{proof}
 
\begin{definition}	
	A generalized random process $s$ constructed with the method of Theorem \ref{coro:CFwT} 
	  is called  a \emph{$\gamma$-order linear process}. We summarize the situation described in Theorem \ref{coro:CFwT} with the (slightly abusive) notation $s = \Lop^{-1} w$.
\end{definition}

\paragraph{Remark.} 
The L\'evy exponent of a S$\alpha$S white noise is $\Psi(\xi) = - C \abs{\xi}^\alpha$ for some constant $C>0$ and thus is $\alpha$-admissible. The construction of a process $s$ such that $\Lop s = w_\alpha$ therefore relies on the existence of a left inverse $\Top$ of $\Lop^*$ that maps continuously $\S(\R^d)$ into $L^{\alpha}(\R^d)$.

\section{Scaling Limits of $\gamma$-Order Linear Processes} \label{sec:theorems}

In this section, we study the statistical behavior of $\gamma$-order linear  processes at coarse- and fine-scales. We recall that for  a generalized random process $s$ and a nonnegative real number $a$, the process $s( \cdot /a)$ is defined by $\langle s(\cdot / a) , \varphi \rangle = a^d \langle s , \varphi ( a \cdot) \rangle$. 
\begin{itemize}
	\item We zoom out the process when $a<1$. In particular, we  consider the limit case $a \rightarrow 0$ and call it the \emph{coarse-scale  behavior} of $s$.  
	\item We zoom in the process when $a>1$. Again, we   pay attention to the limit case $a \rightarrow \infty$, which one call the \emph{fine-scale behavior} of $s$.  
\end{itemize}
In general, we shall see that $s( \cdot / a)$ has no nontrivial limits itself when $a\rightarrow 0/\infty$. However, we will encounter situations where $a^{H} s(\cdot / a)$ has a stochastic limit for some $H \in \R$. When it exists, the coefficient $H$ is unique and determines the renormalization procedure required to observe the convergence phenomenon.

In what follows, we first treat the case of SDEs driven by S$\alpha$S white noises as a preparatory example. Their solutions are actually self-similar and have therefore  straightforward scaling limit behaviors (Section \ref{subsec:SalphaS}). 
We then give sufficient conditions on the L\'evy exponent to determine the coarse- and fine-scales behaviors of $\gamma$-order linear processes. These results are presented in Section {subsec:main} and are the main contributions of this paper. 
 Finally, we question the necessity of our conditions such that the scaling limit exists in Section \ref{sec:necessity}.

	\subsection{Linear Processes Driven by S$\alpha$S White Noises} \label{subsec:SalphaS}

When the white noise is stable, the change of scale has by definition no effect on the noise up to  renormalization. Under reasonable assumptions on the operator $\Lop$, we extend this fact to solutions of SDEs driven by S$\alpha$S white noises. This property is referred to as   self-similarity. 

\begin{definition}
	A generalized random process $s$ is said to be \emph{self-similar of order $H$} if $a^{H} s( \cdot / a ) \overset{(d)}{=} s$ 
	for all $a>0$. The parameter $H$ is called the \emph{Hurst exponent} of $s$. 
\end{definition}

The coarse-scale and fine-scale behaviors of a self-similar process are obvious, since the law of the process is not changed by scaling, up to  renormalization. The self-similarity property is directly inferred from the characteristic functional of the process. Indeed, since 
$\CF_{a^{H} s(\cdot /a) }(\varphi) 
=
 \mathbb{E}\left[ \mathrm{e}^{\mathrm{i} \langle a^{H} s(\cdot /a) ,\varphi \rangle} \right] 
=
\mathbb{E}\left[ \mathrm{e}^{\mathrm{i} \langle s ,a^{d+H} \varphi(a \cdot)  \rangle} \right]  
=
  \CF_s(a^{d+H} \varphi(a \cdot) )$,
we deduce that $s$ is self-similar of order $H$ if and only if
$	\CF_s(\varphi) = \CF_s(a^{d+H} \varphi(a \cdot) )
$ for every $\varphi \in \S(\R^d)$ and $a>0$. This equivalence and some other considerations on self-similar processes can be found in \cite[Section 7.2]{Unser2014sparse}.

\begin{proposition} \label{prop:selfsim}
 	Let $\gamma \geq 0$ and $0<\alpha \leq 2$. We assume that $s = \Lop^{-1} w_\alpha$ is a $\gamma$-order linear process driven by a S$\alpha$S white noise. Then, $s$ is self-similar with Hurst exponent 
	$H = \gamma + d\left( \frac{1}{\alpha} - 1\right)$.
\end{proposition}

\begin{proof}
	By definition of a $\gamma$-order linear process, there exists a linear operator $\Top$, left inverse of $\Lop^*$, that is $(-\gamma)$-homogeneous, continuous from $\S(\R^d)$ to $L^{\alpha}(\R^d)$,  and such that $\CF_s(\varphi) = \exp \left( - C \lVert \Top \varphi \rVert^\alpha_\alpha\right)$. Then, we have
	\begin{align*}
		\CF_s(a^{d+ H } \varphi ( a \cdot) ) &=   \exp \left( - C  \lVert a^{\gamma + d/ \alpha} \Top \{ \varphi (a \cdot) \} \rVert_\alpha^\alpha \right) \overset{(i)}{=} \exp \left( - C   \lVert a^{ d/ \alpha} \{  \Top \varphi \}  (a \cdot) \rVert_\alpha^\alpha \right)  \overset{(ii)}{=} \exp \left( - C \lVert \Top \varphi \rVert_\alpha^\alpha\right)		 = \CF_s(\varphi),  \nonumber
	\end{align*}
	where we used respectively the $(-\gamma)$-homogeneity of $\Top$ and the change of variable $\bm{y} = a \bm{x}$ in (i) and (ii).  This implies that $s$ is self-similar with the  Hurst exponent $H = \gamma + d(1/\alpha - 1)$.
	\end{proof}

	\subsection{Linear Processes at Coarse- and Fine-Scales} \label{subsec:main}

Here we consider the general problem of characterizing the coarse and large scale behaviors of $\gamma$-order linear processes.
We analyse the coarse- and fine-scale behavior separately even if the methods of proof are  similar, in order to emphasize the different  assumptions: the relevant parameter of the underlying white noise is the index $\beta_0$ at coarse scales and $\beta_\infty$ at fine scales.
 
 We have seen in Section \ref{subsec:SalphaS} that two ingredients are sufficient to make a linear process self-similar: the self-similarity of the L\'evy noise and the homogeneity of the left-inverse operator $\Top$ of $\Lop^*$. Moreover, the self-similarity of a L\'evy noise is equivalent to the stability of the underlying infinitely divisible random variable \cite{Taqqu1994stable}. 
Even if $\gamma$-order linear processes are not self-similar in general, one can often recover the self-similarity by asymptotically zooming  the process in or out.  

\begin{definition} \label{def:asslocselfsim}
	We say that the generalized random process $s$ is
	 \emph{asymptotically self-similar of order $H_\infty \in \R$} (\emph{locally self-similar of order $H_{\loc} \in \R$}, respectively)  if the rescaled processes $a^{H_\infty} s(\cdot / a)$ ($a^{H_\loc} s(\cdot / a)$ , respectively) converge in law to a non-trivial generalized random process as $a\rightarrow 0$ ($a\rightarrow \infty$, respectively).
\end{definition}

\noindent The terminology of Definition \ref{def:asslocselfsim} is justified by Proposition \ref{def:asslocselfsim}.

\begin{proposition}  \label{prop:Hinftyforce}
If the generalized random process $s$ is asymptotically self-similar of order $H_\infty$ (locally self-similar of order $H_\loc$, respectively), then the limit is self-similar of order $H_\infty$ ($H_\loc$, respectively).  
\end{proposition}

\begin{proof}	
	Assume that $s$ is asymptotically self-similar of order $H_\infty$. 
	Then, there exists a process $s_{\infty}$ such that $a^{H_\infty} s(\cdot / a)$ converges in law to $s_\infty$. Let $b>0$ and set $s_b = b^{H_\infty} s ( \cdot / b)$. Clearly, $s_b$ is also asymptotically self-similar with limit $b^{H_\infty} s_\infty(\cdot / b)$. 
	Moreover, $a^{H_\infty} s_b( \cdot / a) = (ab)^{H_\infty} s(\cdot / (ab))$. This latter quantity has the same limit in law than $a^{H_\infty}s(\cdot /a)$ as $a \rightarrow 0$, which is $s_\infty$. As a consequence, we have shown that $b^{H_\infty}s_\infty(\cdot / b)$ and $s_\infty$ have the same law for any $b>0$ and the limit is $H_\infty$-self-similar. The proof is identical in the local case.
\end{proof}

One now considers the following question: When is a generalized L\'evy process asymptotically self-similar, when is it locally self-similar, and, if so, what are the asymptotic and local Hurst exponents $H_\infty$ and $H_\loc$? Theorems \ref{theo:coarsescale} and \ref{theo:finescale} answer these two questions. 
 
\begin{theorem}[Coarse-scale behavior of $\gamma$-order linear processes] \label{theo:coarsescale}

We consider  a L\'evy white noise $w$ with L\'evy exponent $\Psi$ and index $\beta_0 \in (0,2]$ and a $\gamma$-homogeneous operator $\Lop$ with $\gamma \geq 0$.
We assume that there exists an operator $\Top$ such that $(\Top , \Psi)$ satisfies one of the following set of conditions.
\begin{itemize}
\item \textbf{Condition (C1).} $\Top$ is a $(-\gamma)$-homogeneous left inverse of $\Lop^*$, continuous from $\S(\R^d)$ to $\mathcal{R}(\R^d)$; or
\item \textbf{Condition (C2).} $\Top$ is a $(-\gamma)$-homogeneous left inverse of $\Lop^*$, continuous from $\S(\R^d)$ to $L^{\beta_0}(\R^d)$ and $\Psi$ is $\beta_0$-admissible.
\end{itemize}
Let $s = \Lop^{-1} w$ be the $\gamma$-order linear process with characteristic functional $\CF_s(\varphi) = \CF_w(\Top \varphi)$.
If $\Psi(\xi) \underset{0}{\sim} -  C \abs{\xi}^{\beta_0}$ for some constant $C>0$, we have the   convergence in law
		\begin{equation} \label{eq:scoarse}
			a^{\gamma + d\left( \frac{1}{{\beta_0}} - 1 \right)} s (\cdot / a) \underset{a \rightarrow 0}{\overset{(d)}{\longrightarrow}} s_{\Lop, {\beta_0}},
		\end{equation}
where  $\Lop s_{\Lop,{\beta_0}} \overset{(d)}{=}  w_{\beta_0}$ is a S$\alpha$S white noise with $\alpha = \beta_0$. 
In particular, $s$ is asymptotically self-similar of order
$	H_{\infty} = \gamma + d\left( \frac{1}{\beta_0} - 1 \right)$.
\end{theorem}

\begin{proof}
	First, assuming that (C1) or (C2) holds, Theorem \ref{coro:CFwT} implies that both $\CF_w(\Top \varphi)$ and \mbox{$\CF_{w_{\beta_0}}(\Top \varphi) = \exp(- C \lVert \Top \varphi \rVert_{\beta_0}^{\beta_0} )$} are  characteristic functionals, hence the processes $s$ and $s_{\Lop, \beta_0}$ are well defined. 

	By Theorem \ref{theo:BoulicautLevy}, we know moreover that the convergence in law \eqref{eq:scoarse} is equivalent to the pointwise convergence of the characteristic functionals. Hence, we have to prove that, for every $\varphi \in \S(\R^d)$,
	\begin{equation} \label{eq:convergenceCFnoise}
		\log \CF_{a^{\gamma + d\left( {1}/{{\beta_0}} - 1 \right)} s(\cdot / a) }(\varphi) \underset{a \rightarrow 0}{\longrightarrow} - C \lVert  \Top \varphi \rVert_{\beta_0}^{\beta_0}. 
	\end{equation}
	We fix $\varphi \in \S(\R^d)$. Then, we have
	\begin{align} 
		\langle a^{\gamma + d\left( {1}/{{\beta_0}} - 1 \right)} s(\cdot / a) , \varphi \rangle 
		&= \langle w, a^{\gamma + d/{\beta_0}} \varphi( a \cdot) \rangle 		\overset{(i)}{=} \langle w , \Top \{a^{\gamma + d / \beta_0} \varphi(a \cdot) \} \rangle 		\overset{(ii)}{=} \langle w , a^{d / \beta_0} \{\Top \varphi \}(a \cdot) \rangle, 
	\end{align}
	where we have used that $\langle s, \varphi \rangle = \langle w , \Top \varphi \rangle$ and the $(-\gamma)$-homogeneity of $\Top$   in (i) and (ii), respectively. Therefore, we have
	\begin{align} \label{eq:intermediaries}
			\log \CF_{a^{\gamma + d\left(  {1}/{{\beta_0}} - 1 \right)} s(\cdot / a) }(\varphi)  & = \log \CF_w(a^{d/{\beta_0}} \{\Top \varphi \} (a \cdot))  = \int_{\R^d} \Psi( a^{d/ {\beta_0}} \{\Top \varphi \} (a \bm{x})) \drm \bm{x} = \int_{\R^d} \left( a^{-d} \Psi( a^{d/{\beta_0}} \Top \varphi (\bm{y}) \right) \drm \bm{y}. 
	\end{align}
	By assumption on $\Psi$, we have moreover that, for every $\bm{y}\in \R^d$,
	\begin{equation}
		a^{-d} \Psi(a^{d/{\beta_0}}  \Top \varphi (\bm{y})) \underset{a\rightarrow 0}{\longrightarrow} - C \abs{ \Top \varphi(\bm{y})}^{\beta_0}.
	\end{equation}
We split here the proof in two parts, depending on whether $\Top$ and $\Psi$ follow (C1) or (C2).
	
	\begin{itemize}
	
	\item We start with (C2).
	The $\beta_0$-admissibility of $\Psi$ implies that
	\begin{equation} \label{eq:boundy}
	\abs{a^{-d} \Psi(a^{d/{\beta_0}} \Top \varphi(\bm{y}))} \leq C' \abs{ \Top \varphi(\bm{y})}^{\beta_0}
	\end{equation}
	for some $C'>0$ and every $\bm{y}\in\R^d$. The right term of \eqref{eq:boundy} is integrable by assumption on $\Top$. Therefore, the Lebesgue   dominated-convergence theorem applies and \eqref{eq:convergenceCFnoise} is showed.
	
	\item We assume now (C1).
	In that case, we do not have a full bound on $\Psi$. 
	We know however that $\Top \varphi$ is bounded, hence $\lVert \Top \varphi \rVert_\infty < \infty$. 
	Since $\Psi$ is continuous and behaves like $(-C \abs{\cdot}^{\beta_0})$ at the origin, there exists $C' > 0$ such that $\abs{\Psi(\xi)} \leq C' \abs{\xi}^{\beta_0}$ for every $\abs{\xi}\leq \lVert \Top \varphi\rVert_\infty$. 
	Hence, for all $a\leq 1$, we have $\abs{a^{d/\beta_0} \Top \varphi(\bm{y})} \leq 1$, and  \eqref{eq:boundy} is still valid. Again, we deduce  \eqref{eq:convergenceCFnoise}  from the Lebesgue dominated convergence theorem. 
	
	\end{itemize}
	Finally, the limit process $s_{\Lop, \beta_0}$ is self-similar with order $H_{\infty} = \gamma + d\left( \frac{1}{\beta_0} - 1 \right)$ according to Proposition \ref{prop:selfsim}. 
	\end{proof}

\begin{theorem}[Fine-scale behavior of $\gamma$-order linear processes] \label{theo:finescale}
Under the same assumptions as in Theorem \ref{theo:coarsescale} but replacing $\beta_0$ by $\beta_\infty \in (0,2]$, we consider $s = \Lop^{-1} w$ a $\gamma$-order linear process. If the L\'evy exponent $\Psi$ of $w$ satisfies $\Psi(\xi) \underset{\infty}{\sim} -  C \abs{\xi}^{\beta_\infty}$ for some constant $C>0$, then we have the   convergence in law 
		\begin{equation} \label{eq:sfine}
			a^{\gamma + d\left( \frac{1}{{\beta_\infty}} - 1 \right)} s (\cdot / a) \underset{a \rightarrow \infty}{\overset{(d)}{\longrightarrow}} s_{\Lop, {\beta_\infty}},
		\end{equation}
where  $\Lop s_{\Lop,{\beta_\infty}} \overset{(d)}{=}  w_{\beta_\infty}$ is a S$\alpha$S white noise with $\alpha = \beta_\infty$.
In particular, $s$ is locally self-similar of order
$	H_{\loc} = \gamma + d\left( \frac{1}{\beta_\infty} - 1 \right).$
\end{theorem}

\begin{proof}
	The proof is similar to the one of Theorem \ref{theo:coarsescale}, and we only develop the parts that differ.
	If $\Top$ and $\Psi$ satisfy (C2), the proof follows exactly the line of Theorem \ref{theo:coarsescale}. We should therefore assume that $\Top$ maps continuously  $\S(\R^d)$ to $\mathcal{R}(\R^d)$. 
	Restarting from \eqref{eq:intermediaries}  with $\beta_\infty$ instead of $\beta_0$, we split the integral into two parts and  get
	\begin{small}
	\begin{align}  \label{eq:decomposeIJ}
		\log \CF_{a^{\gamma + d\left(  {1}/{{\beta_\infty}} - 1 \right)} s(\cdot / a) }(\varphi)  & =   \int_{\R^d} \One_{\abs{ \Top \varphi(\bm{y})} a^{d/{\beta_\infty}} \geq 1} a^{-d} \Psi( a^{d/{\beta_\infty}}  \Top  \varphi(\bm{y}) ) \drm \bm{y} +  \int_{\R^d}  \One_{\abs { \Top  \varphi(\bm{y})} a^{d/{\beta_\infty}} < 1} a^{-d} \Psi( a^{d/{\beta_\infty}} \Top  \varphi(\bm{y}) ) \drm \bm{y} \nonumber  \\
		&:= I(a) + J(a).
	\end{align} 
	\end{small}
		\textit{Control of $I(a)$:} We have, by assumption on $\Psi$, that 
$\One_{\abs{ \Top  \varphi(\bm{y})} a^{d/{\beta_\infty}} \geq 1} a^{-d} \Psi(a^{d/{\beta_\infty}} \Top  \varphi(\bm{y}))  \underset{a\rightarrow \infty}{\longrightarrow} - C \abs{ \Top \varphi(\bm{y})}^{\beta_\infty}.$ 	Moreover, since the continuous function $\Psi$ behaves like $(-C\abs{\cdot}^{\beta_\infty})$ at infinity, there exists a constant $C'$ such that $\abs{\Psi(\xi)} \leq C' \abs{\xi}^{\beta_\infty}$ for every  $\abs{\xi} \geq 1$. Moreover, the function $\Top \varphi$, which is  in $\mathcal{R}(\R^d)$, is bounded. Hence, we have, when $a\geq \lVert \Top \varphi \rVert_\infty^{-1}$, that 
$	\abs{ \One_{\abs{ \Top  \varphi(\bm{y})} a^{d/{\beta_\infty}} \geq 1} a^{-d} \Psi(a^{d/{\beta_\infty}} \Top  \varphi(\bm{y}))} \leq C' \abs{\Top  \varphi(\bm{y})}^{\beta_\infty} $	for all $\bm{y}\in \R^d$. The function on the right is integrable, therefore the Lebesgue dominated convergence applies and we obtain that $I(a) \underset{a\rightarrow \infty}{\longrightarrow} - C \lVert \Top  \varphi \rVert_{\beta_\infty}^{\beta_\infty}$. 
		
\noindent	\textit{Control of $J(a)$:}
	As seen in \eqref{eq:boundCFw1w2}, there exists $C'>0$ and $\epsilon > 0$ such that $\abs{\Psi(\xi)} \leq C' (\abs{\xi}^{\epsilon} + \abs{\xi}^2)$. 
	Without loss of generality, one can choose $\epsilon < \beta_\infty$. 
	Then, for $\abs{\xi} \leq 1$, we have $\abs{\Psi(\xi) }\leq 2 C' \abs{\xi}^\epsilon$ and, therefore, 
	\begin{equation}
		\abs{ \int_{\R^d} 	\One_{\abs{\Top\varphi(\bm{y})} a^{d/{\beta_\infty}} < 1} a^{-d}  {\Psi( a^{d/{\beta_\infty}} \Top \varphi(\bm{y}) )} \drm \bm{y}} \leq  2 C' a^{d( \epsilon / {\beta_\infty} - 1)}  \lVert \Top \varphi \rVert^{\epsilon}_{\epsilon}.
	\end{equation}
	Since $\mathcal{R}(\R^d) \subset L^{\epsilon}(\R^d)$ and $\epsilon < \beta_\infty$, we have $ \lVert \Top \varphi \rVert^{\epsilon}_{\epsilon}< \infty$ and $a^{d( \epsilon / {\beta_\infty} - 1)} \underset{a\rightarrow \infty}{\longrightarrow} 0$, which implies  that $J(a) \underset{a\rightarrow \infty}{\longrightarrow} 0$. We have shown that 
$\log \CF_{a^{\gamma + d\left(  {1}/{{\beta_\infty}} - 1 \right)} s(\cdot / a) }(\varphi) = I(a) + J(a) \underset{a \rightarrow \infty}{\longrightarrow} - C \lVert  \Top \varphi \rVert_{\beta_\infty}^{\beta_\infty}$, 
	as expected. 
	Finally, the limit process $s_{\Lop, \beta_\infty}$ is self-similar with order $H_{\loc} = \gamma + d\left( \frac{1}{\beta_\infty} - 1 \right)$ according to Proposition \ref{prop:selfsim}. 
\end{proof}

	\subsection{Discussion and Converse Results}\label{sec:necessity}
	
	In this section, we investigate the generality of our results by questioning the hypotheses in Theorems \ref{theo:coarsescale} and \ref{theo:finescale}.  One should only consider the case of $\gamma$-homogeneous $\Lop$ operators whose adjoint have a $(-\gamma)$-homogeneous stable inverse $\Top$. We start with preliminary results.
	
	\begin{itemize}
	
	\item The renormalization procedures in Theorems \ref{theo:coarsescale} and \ref{theo:finescale} have to be compared with the index $H = \gamma + d( 1 / \alpha - 1)$ of a $\gamma$-order self-similar process (see Proposition \ref{prop:selfsim}). In particular, the $\gamma$-order linear processes studied in this section are asymptotically  self-similar with index $\gamma + d( 1 / \beta_{0/\infty} - 1)$, where $\beta_{0/\infty} = \beta_0$ or $\beta_\infty$. One can say that the lack of self-similarity of $s$ vanishes asymptotically or locally.
		
	\item (C1) has to be understood as the sufficient assumption  on the operator $\Top$ such that the process $s$ with characteristic functional $\CF_w(\Top \varphi)$ is well defined without any  additional assumption on the L\'evy white noise $w$. Therefore, (C1) is restrictive for the operator but applies to any L\'evy noise.

	\item  This is in contrast to (C2). Here,  the restriction on $\Top$ is minimal since the process $s_{\Lop, \beta_{0/\infty}}$ should be well defined, and, therefore, $\Top$ should at least map $\S(\R^d)$ into $L^{\beta_{0/\infty}}(\R^d)$. It means that (C2) gives  sufficient assumptions on the L\'evy white noise such that the minimal assumption on $\Top$ is also sufficient.

	\item When the variance of the noise is finite, we have in particular that $\beta_0 = 2$. Under the assumptions of Theorem \ref{theo:coarsescale}, the process $a^{\gamma - d/2} s(\cdot / a)$ converges to a Gaussian self-similar process. This can be seen as a central limit theorem for $\gamma$-order finite-variance linear processes. This finite-variance result was already established in our previous work \cite[Theorem 4.2]{Fageot2015wavelet}.  Theorem \ref{theo:coarsescale} is a generalization for the infinite-variance case. 

		\item For important classes of L\'evy white noises, the parameter $\beta_\infty$ is $0$ and Theorem \ref{theo:finescale} does not apply. This includes (generalized) Laplace white noises and compound-Poisson white noises (see Section \ref{subsec:examplesnoises}). In that case, one does not expect the underlying process to be locally self-similar. This is made more precise and proved when $\Lop$ satisfies (C1)  in Proposition \ref{prop:betanull}.
			
\end{itemize}

\begin{proposition} \label{prop:betanull}
	Let $w$ be a white noise with Blumenthal-Getoor index $\beta_\infty = 0$. 
	Assume that $\Lop$ is a \mbox{$\gamma$-homogeneous} operator and that there exists a $(-\gamma)$-homogeneous left inverse $\Top$ of $\Lop^*$ that is continuous from $\S(\R^d)$ to $\mathcal{R}(\R^d)$. Let $s = \Lop^{-1} w$ be the $\gamma$-order linear process with characteristic functional $\CF_s(\varphi) = \CF_w(\Top \varphi)$. Then, for  every $H \in \R$,
$a^H s (\cdot / a) \underset{a \rightarrow \infty}{\overset{(d)}{\longrightarrow}} 0$. 
\end{proposition}

\begin{proof}
	Due to Theorem \ref{theo:BoulicautLevy}, we have to show that, for every $\varphi \in \S(\R^d)$, $\log \CF_{a^H s(\cdot /a)} (\varphi) \underset{a\rightarrow \infty}{\longrightarrow} 0$.
	Proceeding as in Theorem \ref{theo:coarsescale}, we easily show that
	\begin{equation} \label{eq:betainfnullinterm}
	\log \CF_{a^H s(\cdot /a)} (\varphi) = \int_{\R^d} a^{-d} \Psi(a^{d+H} \Top \varphi (\bm{y})) \drm \bm{y}.
	\end{equation}
	According to \eqref{eq:boundCFw1w2}, there exists $\epsilon,C' > 0$ such that $\abs{\Psi(\xi)}\leq C' 	\abs{\xi}^\epsilon$ for $\abs{\xi} \leq 1$. Without loss of generality, one can assume that $\epsilon < \frac{d}{d + \abs{H}}$. This implies in particular that $\epsilon ( d + H) - d < 0$. The knowledge  that $\beta_\infty = 0$ is enough to deduce that $\Psi(\xi)$ is also dominated by $\abs{\xi}^\epsilon$ for $\abs{\xi}\geq 1$. Thus, there exists $C>0$ such that
	$\abs{\Psi(\xi)} \leq C \abs{\xi}^\epsilon$ for every $\xi \in \R$. Restarting from \eqref{eq:betainfnullinterm}, we obtain that
	\begin{equation}
	\abs{\log \CF_{a^H s(\cdot /a)} (\varphi)} \leq C \int_{\R^d} a^{\epsilon(d+H) - d} \abs{\Top \varphi (\bm{y})}^\epsilon \drm \bm{y}  = C \lVert \Top \varphi \rVert_\epsilon^\epsilon a^{\epsilon(d+H)  - d},
	\end{equation} 
	which  vanishes when $a\rightarrow \infty$ due to our choice for $\epsilon$. This concludes the proof. 
\end{proof}

In Theorem \ref{theo:coarsescale} and \ref{theo:finescale}, we assume some asymptotic behaviors of the L\'evy exponent at $0$ or at $\infty$. We see here that under reasonable conditions, this assumption is necessary for \eqref{eq:scoarse} or \eqref{eq:sfine} to occur. 

\begin{proposition}
Let $s = \Lop^{-1} w$ be a $\gamma$-order linear process with $\gamma \geq 0$. We also assume  that $s$ behaves at coarse scale as 
\begin{equation} \label{eq:asymptoticalphastable}
a^{H_\infty} s(\cdot / a) \underset{a\rightarrow 0}{\longrightarrow} s_{\Lop,\alpha}
\end{equation}
 where $H_\infty \in \R$ and $\Lop s_{\Lop,\alpha} = w_\alpha$ is a S$\alpha$S white noise with $0< \alpha \leq 2$. Then, $H_\infty = \gamma + d(1 / \alpha - 1)$.  If moreover $\Psi$ is bounded by $\abs{\cdot}^\alpha$ at the origin, then $\Psi(\xi) \underset{0}{\sim} - C \abs{\xi}^\alpha$ for some $C>0$.  

\noindent Assume now that $s$ behaves at fine scale as 
\begin{equation} \label{eq:localalphastable}
a^{H_\loc} s(\cdot / a) \underset{a\rightarrow \infty}{\longrightarrow} s_{\Lop,\alpha}
\end{equation}
where $H_\loc \in \R$ and $\Lop s_{\Lop,\alpha} = w_\alpha$ is a S$\alpha$S white noise with $0< \alpha \leq 2$. Then, $H_\loc = \gamma + d(1 / \alpha - 1)$.  If moreover $\Psi$ is bounded by $\abs{\cdot}^\alpha$ at $\infty$, then $\Psi(\xi) \underset{\infty}{\sim} - C \abs{\xi}^\alpha$ for some $C>0$.  
\end{proposition}

\begin{proof}
Due to \eqref{eq:asymptoticalphastable}, $s$ is asymptotically self-similar, hence its limit is self-similar of order $H_\infty$ (Proposition \ref{prop:Hinftyforce}). We also now that $s_{\Lop,\alpha}$ is self-similar of order $\gamma + d(1/\alpha-1)$ with Proposition \ref{prop:selfsim}. Thus, $H_\infty = \gamma + d(1/\alpha-1)$.

By $\gamma$-homogeneity, we have $\Lop \{ a^{\gamma + d(1/\alpha -1)} s(\cdot / a) \} = a^{d(1/\alpha -1)} w(\cdot / a)$. Hence, applying the linear operator $\Lop$ each side, \eqref{eq:asymptoticalphastable} implies that $a^{d(1/\alpha - 1)} w(\cdot / a)$ converges in law to $w_\alpha$. In particular, $\CF_w(a^{d/\alpha}\varphi(a\cdot) )$ converges to $\CF_{w_\alpha}(\varphi) = \exp(- C \lVert \varphi \rVert_\alpha^\alpha)$ for $\varphi \in \S(\R^d)$. We show now that this convergence can be extended to functions $f \in \mathcal{R}(\R^d)$. Indeed, for $a >0$, $f \in \mathcal{R}(\R^d)$, $\varphi \in \S(\R^d)$, we have
\begin{small}
\begin{align} \label{eq:decomposepourconverse}
\abs{\CF_w(a^{d/\alpha}f(a\cdot) ) - \CF_{w_\alpha}(f) } 
&\leq 
\abs{\CF_w(a^{d/\alpha} f(a\cdot) ) - \CF_w( a^{d/\alpha} \varphi(a\cdot) ) }
+ 
\abs{\CF_w(a^{d/\alpha} \varphi(a \cdot) ) - \CF_{w_\alpha} (\varphi) }
+
\abs{ \CF_{w_\alpha} (\varphi) -  \CF_{w_\alpha} (f)} \nonumber \\
&= (i) + (ii) + (iii)
\end{align}
\end{small}
Using the arguments of the proof of Proposition \ref{prop:Rcontinuous} and Theorem \ref{theo:coarsescale}, we  see that 
$
(iii)\leq 2 \abs{1 - \CF_{w_\alpha} (f - \varphi) } \leq 2 C \lVert f - \varphi \rVert_\alpha^\alpha
$.
With the same ideas, we have
\begin{align} \label{eq:intermtobound}
	(i) 
	&	\leq 
	2 \abs{1 - \CF_w( a^{d/\alpha} (f(a\cdot ) - \varphi (a \cdot ) ) ) }
	\leq 2 \int_{\R^d} a^{-d} \abs{\Psi( a^{d/\alpha} ( f(\bm{y}) - \varphi(\bm{y}) ) ) } \mathrm{d}\bm{y} 
\leq C' \lVert f - \varphi \rVert_\alpha^\alpha,
\end{align}
where the last inequality is obtained by exploiting that $\abs{\Psi}$ is bounded by $\abs{\cdot}^\alpha$ at the origin.  
The second term $(ii)$ vanishes when $a \rightarrow 0$ for $\varphi \in \S(\R^d)$ fixed. It means it suffices to select $\varphi \in \S(\R^d)$ such that $\lVert f - \varphi \rVert_\alpha^\alpha$ is small and then $a>0$ such that $(ii)$ is small (this is possible because $\varphi \in \S(\R^d)$, hence $(ii)$ vanishes when $a \rightarrow 0$) to make $\abs{\CF_w(a^{d/\alpha}f(a\cdot) ) - \CF_{w_\alpha}(f) } $ arbitrarily small as expected. 

\noindent Let us now consider $f = \One_{[0,1]^d}$. Then, we have
$\log \CF_w(a^{d/\alpha} f(a\cdot ) = a^{-d} \Psi(a^{d/\alpha}) \underset{a\rightarrow 0 }{\longrightarrow} \log \CF_{w_\alpha}(f) = - C$. With $ f = - \One_{[0,1]^d}$, we have similarly that $a^{-d} \Psi( - a^{d/\alpha})$ converges to $-C$. Finally, setting $\xi= \pm a^{d/\alpha}$, we obtain that $\Psi(\xi) \underset{\abs{\xi}\rightarrow 0}{\sim} - C \abs{\xi}^\alpha$.

\noindent The proof for the local case is very similar. The only difference is for the control of $K(a):=\int_{\R^d} a^{-d} \abs{\Psi( a^{d/\alpha} ( f(\bm{y}) - \varphi(\bm{y}) ) ) } \mathrm{d}\bm{y}$ in \eqref{eq:intermtobound} when $a\rightarrow \infty$. Then, the result follows from the same decomposition and arguments used in \eqref{eq:decomposeIJ}. Indeed, we have
\begin{small}
\begin{align}
K(a) &= \int_{\R^d} \One_{ a^{d/\alpha} \abs{ f(\bm{y}) - \varphi(\bm{y}) } \geq 1} a^{-d} \abs{\Psi( a^{d/\alpha} ( f(\bm{y}) - \varphi(\bm{y}) ) ) } \mathrm{d}\bm{y}  + \int_{\R^d} \One_{ a^{d/\alpha} \abs{f(\bm{y}) - \varphi(\bm{y}) } < 1} a^{-d} \abs{\Psi( a^{d/\alpha} ( f(\bm{y}) - \varphi(\bm{y}) ) ) } \mathrm{d}\bm{y}\nonumber \\
& := I(a) + J(a).
\end{align}
\end{small}
We bound $I(a) \leq C \lVert f - \varphi\rVert_\alpha^\alpha$ because $\Psi$ is bounded by $\abs{\cdot}^\alpha$ at infinity. We also have that $J(a)$ vanishes when $a \rightarrow \infty$, as we see by bounding $\abs{\Psi}$ by $\abs{\cdot}^\epsilon$ with $\epsilon < \alpha$ ($\epsilon$ exists because $w$ is tempered, see \eqref{eq:boundCFw1w2}). 
\end{proof}

As a final remark, we point out that there exist L\'evy exponents $\Psi$ that are oscillating between two different power laws at infinity. Some examples are constructed in \cite[Examples 1.1.15 and  1.1.16]{Farkas2001function}. These examples coupled with Proposition \ref{eq:asymptoticalphastable} implies that one cannot hope to have local self-similarity for any $\gamma$-order linear process.

\section{Application to specific classes of SDEs and Simulations} \label{sec:examples}

	\subsection{Examples of L\'evy White Noises} \label{subsec:examplesnoises}
	
	We introduce  classical families  of L\'evy white noises that  allow us to illustrate our results.
	
	\paragraph{From infinitely divisible random variables to  L\'evy white noises.}
	Consider a L\'evy white noise $w$ on $\S'(\R^d)$ and a family of functions $\varphi_n \in \S(\R^d)$ that converges to $\One_{[0,1]^d}$ for the topology of $\mathcal{R}(\R^d)$ (see Section \ref{subsec:operators}). Since the characteristic functional of $w$ is continuous over $\mathcal{R}(\R^d)$ (Proposition \ref{prop:Rcontinuous}), one can  show that the sequence $(\langle w ,\varphi_n\rangle)$ is a Cauchy sequence in $L^0(\Omega)$. It therefore  converges   to some random variable denoted by $X = \langle w, \One_{[0,1]^d} \rangle$. This random variable is infinitely divisible,  with  characteristic function
	\begin{equation}  \label{eq:CFXxi}
		\CF_X(\xi) =   \exp \left( \int_{\R^d} \Psi ( \xi  \One_{[0,1]^d} (\bm{x})) \drm \bm{x} \right) = \exp \left( \Psi(\xi )\right).
	\end{equation}
	The latter equality in \eqref{eq:CFXxi} comes from the fact that $\Psi(0) = 0$.
The law of $w$ is fully characterized by the law of $\langle w, \One_{[0,1]^d} \rangle$. This principle is made rigorous and extended to many more test functions in \cite{Fageot2017unified}.

By convention, the terminology for the random variable $\langle w, \One_{[0,1]^d} \rangle$ is inherited by  the underlying white noise $w$. We have already exploited this principle for the definition of S$\alpha$S white noises, with the particular case of the Gaussian white noise. Another example is the \emph{Cauchy white noise}, that corresponds to the case $\alpha=1$. 

\paragraph{Compound-Poisson White Noises.}
A \emph{compound-Poisson white noise} is such that $\langle w, \One_{[0,1]^d} \rangle$ is a compound-Poisson random variable with characteristic function of the form \cite[Section 4.4.2]{Unser2014sparse} 
$\CF_{\langle w, \One_{[0,1]^d} \rangle}(\xi) = \exp ( \lambda ( \CF_{\mathrm{Jump}}(\xi) - 1) )$,
 where  $\lambda >0$ and $\CF_{\mathrm{Jump}}$ is the  characteristic function of a probably law $\mathscr{P}_{\mathrm{Jump}}$ such that $\mathscr{P}_{\mathrm{Jump}} \{0\}= 0$ (no singularity at the origin). The notation $\CF_{\mathrm{Jump}}$ is motivated by the fact that the underlying probability law is the common law of the jumps of the compound Poisson white noise \cite{Unser2011stochastic}. The L\'evy exponent of a compound Poisson white noise is bounded, hence its index is $\beta_\infty = 0$. The Pruitt index $\beta_0$ can take any value in  $(0,2]$ and is equal to $2$ if $\langle w, \One_{[0,1]^d} \rangle$ is symmetric with a finite variance.
 When the law of the jump is Gaussian (Cauchy, respectively), we call $w$ a \emph{Poisson-Gaussian} white noise (a \emph{Poisson-Cauchy} white noise, respectively).

\paragraph{Generalized Laplace White Noises.} 
 Another interesting infinitely divisible family is given by the generalized-Laplace laws. We follow here the notations of \cite{Koltz2001laplace}.  A \emph{generalized-Laplace white noise} is such that $\langle w, \One_{[0,1]^d} \rangle$ is a generalized Laplace variable whose characteristic function is given by $\CF_{\langle w, \One_{[0,1]^d} \rangle} (\xi) = \frac{1}{( 1 + \xi^2)^c} = \exp \left( - c \log ( 1 + \xi^2 ) \right)$ with $c>0$. When $c=1$, we recognize the Laplace law. The Blumenthal-Getoor and Pruitt indices of generalized Laplace white noises are  $\beta_\infty = 0$ (since $\Psi$ grows asymptotically slower than any polynomial) and $\beta_0 = 2$ (symmetric finite-variance white noise), respectively. 
 
\paragraph{Layered Stable White Noises.}
 Finally, we consider the family of white noises introduced by Houdr\'e and Kawai in \cite{Houdre2007layered} to illustrate the richness of the L\'evy family. 
 We first need some notation. 
 A \emph{L\'evy measure} is a measure $\nu$ on $\R$ such that $\nu\{0\} = 0$ and $\int_{\R} \inf(1 , t^2) \nu(\drm t) < \infty$. Then, for $\nu$ a symmetric L\'evy measure, the function 
 $\Psi(\xi) = - \int_{\R} (1 - \cos (\xi t)) \nu (\drm t)$ is a  L\'evy exponent. This is a particular case of the L\'evy-Khintchine decomposition of  a  L\'evy exponent \cite[Theorem 4.2]{Unser2014sparse}. Then, for $\alpha, \beta \in (0,2)$, we consider the  measure
 \begin{equation}
 	\nu_{\alpha,\beta}(\drm t) = \One_{\abs{t} \leq 1} \frac{\drm t}{\abs{t}^{\alpha +1 }} + \One_{\abs{t} > 1} \frac{\drm t}{\abs{t}^{\beta +1 }}.
 \end{equation}
 We easily check that $\nu_{\alpha,\beta}$ is a symmetric L\'evy measure and define therefore the   L\'evy exponent
$ 	\Psi_{\alpha,\beta} (\xi) = - \int_{\R} (1 - \cos(\xi t)) \nu_{\alpha,\beta}(\drm t). $ When $\alpha = \beta$, we recover a S$\alpha$S white noise with L\'evy measure  $\nu_{\alpha} (\drm t) = \drm t / {|t|^{\alpha + 1}}$.
 The L\'evy white noise with exponent $\Psi_{\alpha,\beta}$ is called an \emph{layered stable white noise}. Its  interest for our purpose is that it displays all the possible joint behaviors of the L\'evy exponent at the origin and at infinity, as shown in Proposition \ref{prop:psialphabeta}. Many additional properties of layered stable laws and processes have been studied in \cite{Houdre2007layered}.
 
 \begin{proposition} \label{prop:psialphabeta}
	For $0<\alpha,\beta<2$, the L\'evy exponent $\Psi_{\alpha,\beta}$ satisfies
${\Psi_{\alpha,\beta}(\xi)}  \underset{\infty}{\sim} - C_\infty \abs{\xi}^{\alpha}$ and $
		  {\Psi_{\alpha,\beta}(\xi)}   \underset{0}{\sim} - C_0 \abs{\xi}^{\beta}$
	with $C_0 , C_\infty > 0$ some constants.
 \end{proposition}
 
 \begin{proof}
 	We have
	\begin{equation}
		\Psi_{\alpha,\beta}(\xi) =  - \int_{\abs{t}\leq 1} (1 - \cos(\xi t)) \frac{\drm t}{\abs{t}^{\alpha+1}} - \int_{\abs{t}> 1} (1 - \cos(\xi t)) \frac{\drm t}{\abs{t}^{\beta+1}} := \Psi_1(\xi) + \Psi_2(\xi).
	\end{equation}	
	 Then, by the  change of variable $x = \xi t$, we have  that 
	\begin{equation}
	\Psi_1(\xi) = - \left( \int_{\abs{x} \leq  \abs{\xi}} ( 1 - \cos x) \frac{\drm x}{\abs{x}^{\alpha+1}}  \right) \abs{\xi}^{\alpha} \underset{\infty}{\sim} - \left( \int_{\R} ( 1 - \cos x) \frac{\drm x}{\abs{x}^{\alpha+1}}  \right) \abs{\xi}^{\alpha}
	\end{equation} 
	while $\abs{\Psi_2(\xi)} \leq \int_{\abs{t} > 1} 2 \frac{\drm t}{\abs{t}^{\beta + 1}} = o (\abs{\xi}^{\alpha})$, implying the expected asymptotic behavior with $C_\infty = \int_{\R} ( 1 - \cos x) \frac{\drm x}{\abs{x}^{\alpha+1}} $.
	 Similarly, we have  that 
$ \Psi_2(\xi) \underset{0}{\sim} - \left( \int_{\R} (1 - \cos x) \frac{\drm x}{\abs{x}^{\beta+1}} \right) \abs{\xi}^{\beta}$ while $ \abs{\Psi_1(\xi)} \leq \frac{1}{2} \left( \int_{\abs{t}\leq 1} \frac{\drm t}{\abs{t}^{\alpha-1}} \right) \abs{\xi}^2 = o (\abs{\xi}^{\beta})$, where we have used that $\abs{1-\cos(\xi t)} \leq \frac{\xi^2 t^2}{2}$. This implies the behavior of $\Psi$ at the origin  with $C_0 =  \int_{\R} (1 - \cos x) \frac{\drm x}{\abs{x}^{\beta+1}} $.
 \end{proof}
 	
	Proposition \ref{prop:psialphabeta}  implies that $(\beta_\infty , \beta_0) = (\alpha, \beta)$. 
 	Therefore,  $\gamma$-order linear  processes based on a layered stable white noise share the interesting following property: While failing to be self-similar, they offer a transition from a local self-similarity of order $H_{\loc} = \gamma + d( 1 / \alpha - 1) $   to an asymptotic self-similarity of order $H_{\infty} = \gamma + d( 1 / \beta - 1 )$. This can be of interest for modeling purposes.
 
\paragraph{Summary.}
By studying the behavior  of the L\'evy exponent around the origin and at $\infty$ (as we did for $\Psi_{\alpha,\beta}$), one easily obtains the indices of the L\'evy white noises of Table \ref{table:noises}. 		
\begin{table} [h!] \label{table:noises}
\footnotesize  
\centering
\caption{Some L\'evy white noises with their Blumenthal-Getoor indices.}
\begin{tabular}{lcccccc} 
\hline
\hline 
White noise & Parameter & $\Psi(\xi)$ &  $\beta_0$ & $\beta_\infty$  \\
\hline\\[-1ex]
Gaussian & $\sigma^2 >0$ & $- \sigma^2 \xi^2 / 2$   & $2$ & $2$   \\
Non-Gaussian S$\alpha$S  & $\alpha \in (0,2)$  & $-\lvert \xi \rvert^\alpha$ & $\alpha$ & $\alpha$ \\
Generalized Laplace & $c>0$ & $- c \log(1 + \xi^2)$ & $2$ & $0$ \\
Symmetric finite-variance compound Poisson  & $\lambda >0,\mathscr{P}_{\mathrm{Jump}}$ & $\lambda  (\CF_{\mathrm{Jump}} (\xi) - 1)$  & $2$ & $0$ \\
Compound Poisson with S$\alpha$S jumps  & $\lambda > 0, \alpha \in(0,2)$ & $\lambda ( \mathrm{e}^{-\abs{\xi}^\alpha} - 1)$  & $\alpha$   & $0$   \\
Layered stable & $\alpha,\beta \in (0,2)$ & $\Psi_{\alpha,\beta}(\xi)$  & $\alpha$ & $\beta$ \\
\hline
\hline
\end{tabular} \label{table:noises}
\end{table}

	\subsection{L\'evy Processes and Sheets} \label{subsec:levysheets}

	The canonical basis of $\R^d$ is $(\bm{e}_k)_{k=1\ldots d}$. We denote by $\Der_{k}$ the partial derivative along the direction $\bm{e}_k$.
	Then, a \emph{L\'evy sheet} in dimension $d$ is a solution of 
	\begin{equation}
			\Lop s = \Der_{1} \cdots \Der_{d} s = w
	\end{equation}
	with $w$ a $d$-dimensional L\'evy white noise \cite{Dalang1992}. When $d=1$, one recognizes  the family of L\'evy processes that corresponds to the differential equation $\Der s = w$ in dimension $d=1$.
	
	The linear operator $\Lop   = \Der_{1} \cdots \Der_{d}$ is continuous from $\S(\R^d)$ to $\S(\R^d)$ and $d$-homogeneous. Its adjoint  $\Lop^* = (-1)^d \Der_{1} \cdots \Der_{d}$ admits the  natural  $(-d)$-homogeneous (left and right) inverse defined by
$	(\Lop^*)^{-1} \varphi (\bm{x}) = \int_{(-\infty, x_1) \times \cdots \times (- \infty , x_d)} \varphi(\bm{t})\drm \bm{t}$	for $\bm{x} = (x_1, \ldots , x_d)$ and $\varphi \in \S(\R^d)$. Unfortunately, $(\Lop^*)^{-1}$ is unstable in the sense that it does not map $\S(\R^d)$ in any $L^p (\R^d)$ space, $0<p\leq 2$ (and, \emph{a fortiori}, not in $\mathcal{R}(\R^d)$).
		We can however correct $(\Lop^*)^{-1}$ to transform it into a stable \textit{left} inverse. For this, we define $\Top$ as the adjoint of  the operator 
$			\Top^*  \varphi (\bm{x}) = \int_{(0, x_1) \times \cdots \times (0 , x_d)} \varphi(\bm{t})\drm \bm{t}.
$		The operator $\Top$ is $(-d)$-homogeneous and continuous from $\S(\R^d)$ to $\mathcal{R}(\R^d)$ \cite[Section 4.2]{Fageot2014}. 
		We satisfy therefore the Condition  (C1) of Theorem \ref{coro:CFwT} and define $s =  (\Der_1 \cdots \Der_d)^{-1}w$ with characteristic functional $\CF_s(\varphi) = \CF_w(\Top \varphi)$ for any white noise $w$.	 
		
		This way of defining $s$ can be interpreted in terms of boundary conditions---it   imposes that $s(\bm{x}) = 0$ almost surely for every $\bm{x} = (x_1, \ldots , x_d)$ such that one of the $x_k$ is $0$. In particular, in dimension $d=1$, it imposes that $s(0)=0$ almost surely. 	Here,  the random variable $s(\bm{x})$, not well-defined from the specification of $s$ as a generalized random process, is understood as the limit in probability of random variables $\langle s, \varphi \rangle$ where $\varphi$ approximates the shifted Dirac impulse $\delta(\cdot - \bm{x})$ in an adequate sense. This extension is possible for the same reasons one can consider the random variable $\langle w , \One_{[0,1]^d} \rangle$ for any L\'evy noise, what has been already discussed below. 
	Applying the results of Section \ref{subsec:main}, we directly deduce Proposition \ref{prop:sheets}.

	\begin{proposition} \label{prop:sheets}
	Consider $w$ a L\'evy white noise with indices $0<\beta_0, \beta_\infty\leq 2$, and $s = (\Der_1 \cdots \Der_d)^{-1} w$ as above. Then,
	\begin{itemize}
		\item if $\Psi(\xi) \underset{0}{\sim} - C \abs{\xi}^{\beta_0}$ for some $C>0$, then $a^{d / \beta_0} s(\cdot / a) \underset{a \rightarrow 0}{\longrightarrow} s_{ \Der_1 \cdots \Der_d, \beta_0}$;
	\item if $\Psi(\xi) \underset{\infty}{\sim} - C \abs{\xi}^{\beta_\infty}$ for some $C>0$, then $a^{d / \beta_\infty} s(\cdot / a) \underset{a \rightarrow \infty}{\longrightarrow} s_{ \Der_1 \cdots \Der_d, \beta_\infty}$.
	\end{itemize}
	Here, $s_{\Der_1,\ldots ,\Der_d, \alpha} =  (\Der_1 \cdots \Der_d)^{-1}w_\alpha$, where $w_\alpha$ is a S$\alpha$S white noise. 
	\end{proposition}

We illustrate our results on dimension $1$ with some simulations of L\'evy processes. 
First, we consider three L\'evy processes driven respectively by the Laplace white noise, the Poisson-Gaussian white noise, and the Poisson-Cauchy white noise.
We look at the processes at three different scales by representing them on $[0,1]$,   $[0,10]$, and   $[0,1000]$. 
We only generate one process of each type and represent it on the different intervals: this corresponds to zooming out it.
The theoretical prediction at large scale is as follows: the Laplace and Poisson-Gaussian processes should be statistically indistinguishable from the Brownian motion, while the Poisson-Cauchy process should be statistically indistinguishable from the Cauchy process (also called L\'evy flight).
We see in Figure \ref{fig:largeScale} that this is observed on simulations. 
For comparison purposes, we also represent one realization of the expected limit process.

\begin{figure}[t]
\centering
	\includegraphics[width=1\textwidth]{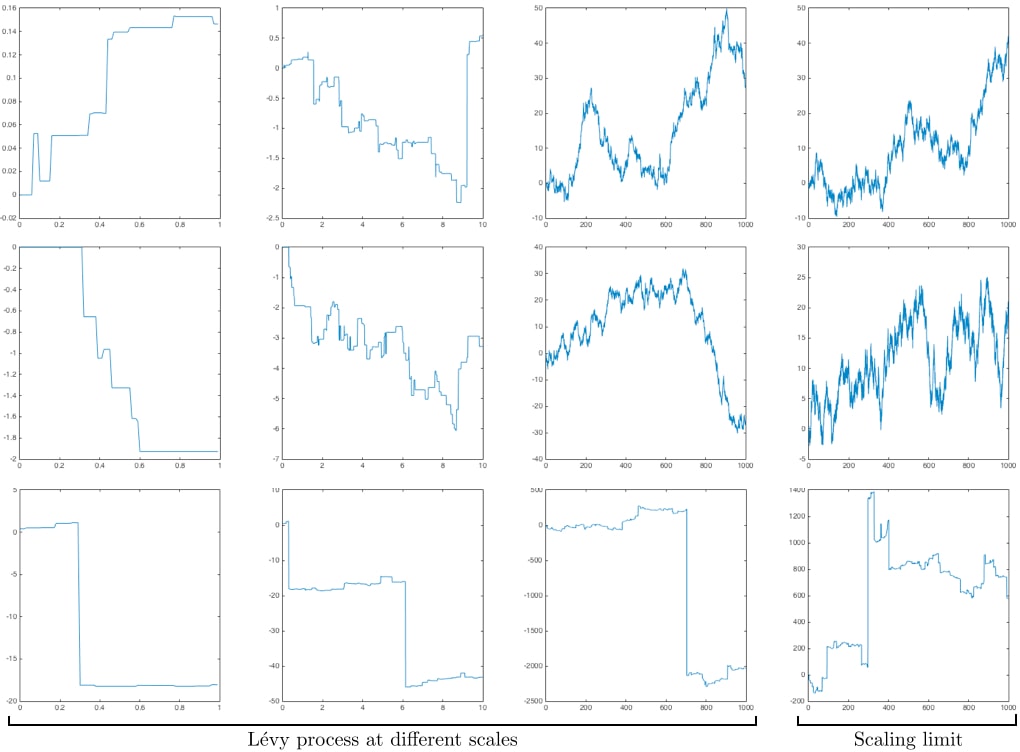}
	\caption{L\'evy processes at three different scale and comparison with the corresponding self-similar process at large scale according to Theorem \ref{theo:coarsescale}.}
	\label{fig:largeScale}
\end{figure}

We now illustrate the difference between fine-scale and coarse scale behaviors. To do so, we consider a L\'evy white noise $w$, sum of a Gaussian and a Cauchy white noise that are independent. Then, we have $\beta_0 = 1$ and $\beta_\infty = 2$. The prediction is that the L\'evy process driven by $w$ converges to the Brownian motion at fine scales and to the L\'evy flight at coarse scales. Again, the theoretical prediction is observed on simulations on Figure \ref{fig:finecoarse}, where one realization of the process is represented on $[0,1/10]$ (fine scale), $[0,10]$ (intermediate scale), and $[0,1000]$ (coarse scale).

\begin{figure}[t]
\centering
	\includegraphics[width=1\textwidth]{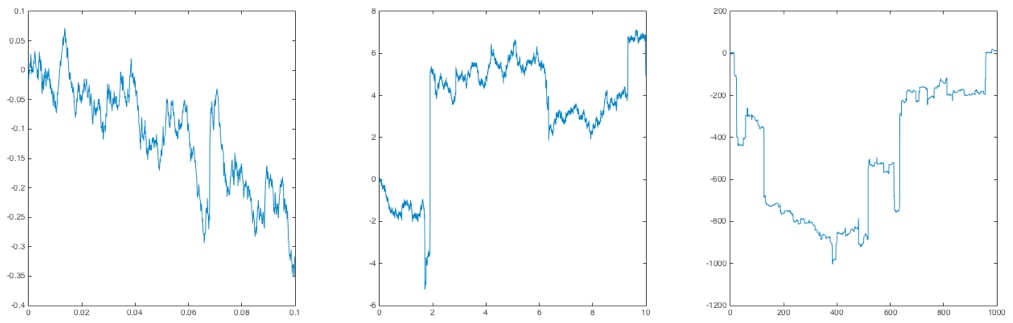}
	\caption{Sum of a L\'evy flight and a Brownian motion at thee different scales.}
	\label{fig:finecoarse}
\end{figure}

	\subsection{Fractional L\'evy Processes and Fields} \label{subsec:fractionallevy}

	In dimension $d$,  we consider the stochastic differential equation
	\begin{equation}
		\Lop s = \FL s = w,
	\end{equation}	
	where $\FL$ is the fractional Laplacian whose Fourier multiplier is $\lVert \bm{\omega} \rVert^\gamma$ with $\gamma \geq 0$ and $\gamma / 2 \notin \N$. 
	The fractional Laplacian is self-adjoint and $\gamma$-homogeneous. For $(p,\gamma)$ satisfying
	\begin{equation} \label{eq:conditionFL}
	p \geq 1 \text{ and } (\gamma + d/p - 1) \notin \N,
	\end{equation}
$\FL$ admits a  (unique)   $(-\gamma)$-homogeneous left inverse $\Top_{\gamma,p}$ that continuously map $\S(\R^d)$ into $L^p(\R^d)$ \cite[Theorem 3.7]{Sun-frac}. For such $p$, if the L\'evy white noise is \mbox{$p$-admissible}, we satisfy Condition  (C2) of Theorem  \ref{coro:CFwT} and define $s =  (\FL)^{-1}w$ with characteristic functional $\CF_s(\varphi) = \CF_w(\Top_{\gamma,p} \varphi)$. The process $s$ is called a fractional L\'evy process (a fractional L\'evy field when $d\geq 2$). 	
	Again, the direct application of the results of Section \ref{subsec:main} yields  Proposition \ref{prop:gamma}.
	
	\begin{proposition} \label{prop:gamma}
	For $(p,\gamma)$ satisfying \eqref{eq:conditionFL}, consider  a $p$-admissible L\'evy white noise $w$ with indices $0<\beta_0, \beta_\infty\leq 2$, and $s = (\FL)^{-1} w$ as above. Then,
	\begin{itemize}
		\item if $\Psi(\xi) \underset{0}{\sim} - C \abs{\xi}^{\beta_0}$ for some $C>0$, then $a^{\gamma + d (1 / \beta_0-1)} s(\cdot / a) \underset{a \rightarrow 0}{\longrightarrow} s_{\FL, \beta_0}$;
	\item if $\Psi(\xi) \underset{\infty}{\sim} - C \abs{\xi}^{\beta_\infty}$ for some $C>0$, then $a^{\gamma + d(1  / \beta_\infty -1)} s(\cdot / a) \underset{a \rightarrow \infty}{\longrightarrow} s_{ \FL, \beta_\infty}$.
	\end{itemize}
	Here, $s_{\FL, \alpha} =  (-\Delta)^{-\gamma / 2}w_\alpha$, where $w_\alpha$ is a S$\alpha$S white noise. 
	\end{proposition}
	
	In dimension $d=1$, identical results can be derived for the fractional derivative $\Lop = \Der^\gamma$ in a very similar fashion. This includes in particular the fractional Brownian motions \cite{Mandelbrot1968} and its L\'evy-driven generalizations. The construction of stable inverses of the adjoint of $\Der^\gamma$ is the subject of \cite[Section 5.5.1]{Unser2014sparse}. 

\section*{Acknowledgements}
The authors are grateful to Thomas Humeau for the fruitful discussions that lead to this work, in particular concerning Proposition \ref{prop:Rcontinuous}.
We also warmly thank Virginie Uhlmann for her help for the simulations.
The research leading to these results was funded by the ERC grant agreement No 692726 - FUN-SP.

\bibliographystyle{plain}
\bibliography{references}

\end{document}